\documentclass[envcountsect]{svjour3}
\usepackage{graphicx,color}
\usepackage[all]{xy}
\usepackage{amssymb,amsmath}

\spnewtheorem*{Theorem}{Theorem 4.2}{\bf}{\it}
\allowdisplaybreaks

\begin{document}

\title{Pillar switchings and acyclic embedding in mapping class group}

\author{Chan-Seok Jeong \and Yongjin Song}

\institute{Chan-Seok Jeong \at Departments of Mathematics, Inca University, Incheon 402-751, Korea  \\ \email{csjeong@inha.ac.kr}\at
\and Yongjin Song \\ Departments of Mathematics, Inca University, Incheon 402-751, Korea \\ \email{yjsong@inha.ac.kr}}

\date{Received: date / Accepted: date}

\maketitle

\begin{abstract}
 The braid group $B_g$ is embedded in the ribbon braid group that is defined to be the mapping class group $\Gamma_{0,(g),1}$. By gluing two copies of surface $S_{0,g+2}$ along $g+1$ holes, we get  surface $S_{g,1}$. A pillar switching is a self-homeomorphism of $S_{g,1}$ which switches two pillars of surfaces by $180{}^\circ$ horizontal rotation.  We analyze the actions of pillar switchings on $\pi_1 S_{g,1}$ and then give concrete expressions of pillar switchings in terms of standard Dehn twists. The map $\psi: B_g \rightarrow \Gamma_{g,1}$ sending the generators of $B_g$ to pillar switchings on $S_{g,1}$ is defined by extending the embedding $B_g \hookrightarrow \Gamma_{0,(g+1),1}$. We show that this map is injective by analyzing the actions of pillar switchings on $\pi_1 S_{g,1}$. The second part of this paper is to prove that this map induces a trivial homology homomorphism in the stable range. For the proof we use the categorical delooping. We construct a suitable monoidal 2-functor from tile category to surface category and show that this functor thus induces a map of double loop spaces.
\end{abstract}


\section{Introduction}

Let $S_{g,k}$ be an orientable surface of genus $g$ obtained from a closed surface by removing $k$ open disks. The mapping class group $\Gamma_{g,k}$ is the group of isotopy classes of orientation preserving self-diffeomorphisms of $S_{g,k}$ fixing the boundary of $S_{g,k}$ that consists of $k$ disjoint circles. Let ${\rm Diff}^+(S_{g,k})$ be the group of orientation preserving self-diffeomorphisms of $S_{g,k}$. The base point component of ${\rm Diff}^+(S_{g,k})$ is contractible and $\Gamma_{g,k}$ is isomorphic to $\pi_0 {\rm Diff}^+(S_{g,k})$.

There are standard Dehn twists generating $\Gamma_{g,1}$.

\begin{figure}[hbt]
  \includegraphics[width=0.55\textwidth]{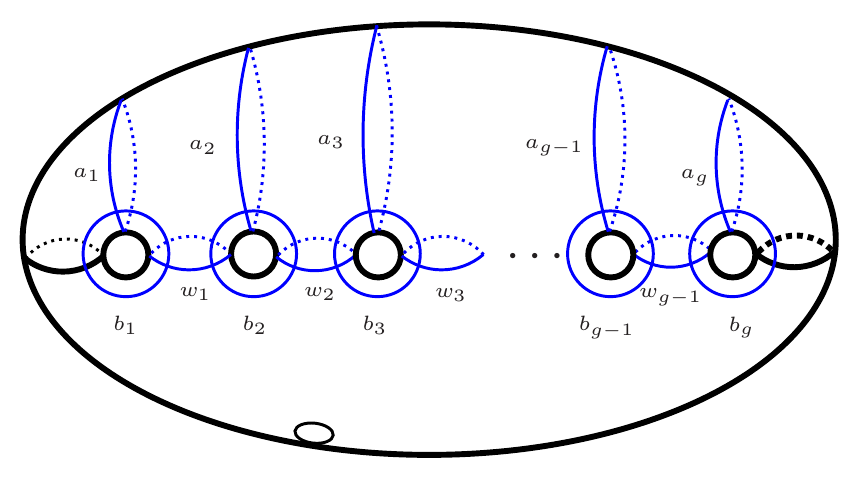}\\
\caption{Standard Dehn twists of $\Gamma_{g,1}$}\label{Dehn twists}
\end{figure}

Let $i: S_{g,r}\rightarrow S_{g,r+1}$ $(r \ge 1)$ and $j:S_{g,r}\rightarrow S_{g+1,r-1}$ $(r\ge 2)$ be the inclusions defined by adding a pair of pants (a copy of $S_{0,3}$) sewn along one boundary component for $i$ and two boundary components for $j$. Also define $l:S_{g,r+1}\rightarrow S_{g,r}$ by gluing a disk to a boundary component.
The maps $i,j,l$ induce maps on the mapping class groups
$$ I:\Gamma_{g,r}\rightarrow\Gamma_{g,r+1},~ J:\Gamma_{g,r}\rightarrow\Gamma_{g+1,r-1},~ L:\Gamma_{g,r+1}\rightarrow\Gamma_{g,r}$$
by extending diffeomorphisms to be the identity on the added pairs of pants and disk. The maps $L:\Gamma_{g,r}\rightarrow\Gamma_{g,r-1}$ and $I\circ J=:\Gamma_{g,r}\rightarrow\Gamma_{g+1,r}$ induce isomorphisms on integral homology in degrees less than $(g-1)/2$. This is Harer's homology stability theorem (\cite{Harer1}) which was improved by Ivanov (\cite{Ivanov1}).  That is we have:

\begin{corollary} $H_*(\Gamma_{g,r};\mathbb Z)$ is independent of $g$ and $r$ in degree $*<(g-1)/2$.
\end{corollary}

In this range, for $\Gamma_{\infty}={\lim}_{g\to\infty}\Gamma_{g,1}$, we have isomorphisms:
\[
H_*(\Gamma_{g,r};\mathbb Z)\stackrel{\simeq}\longrightarrow H_*(\Gamma_{\infty};\mathbb Z)\quad\mbox{ for }~ g\ge 2*+1.
\]

Recently Boldsen, O. Randal-Williams, and N. Wahl \cite{Boldsen,Randal-Williams,Wahl} improved the results of Harer and Ivanov.

\begin{theorem}[\cite{Boldsen,Randal-Williams,Wahl}]
Let $S_{g,r}$ be a surface of genus $g$ with $r$ boundary components.
\begin{itemize}
\item[\rm(a)] Let $r\ge 1$. For $I:\Gamma_{g,r}\rightarrow\Gamma_{g,r+1}$,
$$I_*:H_k(\Gamma_{g,r})\rightarrow H_k(\Gamma_{g,r+1})$$
is an isomorphism for $2g\ge 3k$ and a monomorphism in all degrees.
\item[\rm(b)] Let $r\ge 2$. For $J:\Gamma_{g,r}\rightarrow\Gamma_{g+1,r-1}$,
$$J_*:H_k(\Gamma_{g,r})\rightarrow H_k(\Gamma_{g+1,r-1})$$
is surjective for $2g\ge 3k-1$, and an isomorphism for $2g\ge 3k+2$.
\item[\rm(c)] For $L:\Gamma_{g,1}\rightarrow \Gamma_{g}$,
$$L_*:H_k(\Gamma_{g,1})\rightarrow H_k(\Gamma_{g})$$
is surjective for $2g\ge 3k-1$, and an isomorphism for $2g\ge 3k$.
\end{itemize}
\end{theorem}

For $g>3$, $\Gamma_{g,k}$ is perfect (\cite{Powell}). Apply Quillen's plus construction to the classifying space of it, then we have a simply connected space $B\Gamma_{g,k}^{+}$ and a homology equivalence  $$B\Gamma_{g,k} \longrightarrow B\Gamma_{g,k}^{+}.$$

The direct limit ${\lim}_{g\to\infty}\Gamma_{g,k}$ is the stable mapping class group $\Gamma_{\infty, k}$. The homology stability theorem implies that $B\Gamma_{\infty, k}^{+}$ is independent of $k$ up to homotopy, because it is simply connected. We write $B\Gamma_{\infty}^{+}$ for the common homotopy type.

Let $B_n$ denote the Artin's braid group of $n$-strings with canonical generators $\beta_1,\ldots,\beta_{n-1}$ which have the following presentations:
\begin{align*}
\beta_i\beta_j&=\beta_j\beta_i \quad\mbox{if } |i-j|>2,\\
\beta_i\beta_{i+1}\beta_i&=\beta_{i+1}\beta_i\beta_{i+1} \quad\mbox{for } i=1,\ldots,n-2.
\end{align*}
The second equality is called a braid relation.

Dehn twists, which are standard generators of $\Gamma_{g,1}$, satisfy braid relations. Two Dehn twists $\alpha$, $\beta$ satisfy $\alpha\beta\alpha=\beta\alpha\beta$ if two closed curves representing them intersect at one point. Since there are in $\Gamma_{g,1}$ plenty of braid relations among Dehn twists, there are many embeddings of the braid group into $\Gamma_{g,1}$ via Dehn twists. In this paper we deal with a natural, rather geometric embedding of the braid group into $\Gamma_{g,1}$ via pillar switchings, not via Dehn twists.

Let $\Gamma_{g,(n),1}$ be the mapping class group of surface $S_{g,n+1}$ which consists of isotopy classes of orientation preserving self-homeomorphisms which may permute $n$ boundary components while fix one boundary component pointwise. As a special case, $\Gamma_{0,(n),1}$ is identified with ribbon braid group with $n$ ribbons. The braid group $B_n$ is naturally embedded in $\Gamma_{0,(n),1}$. The surface $S_{0,n+1}$ is regarded as a disk with $n$ holes. The generator $\beta_i$ of $B_n$, which is a crossing of the $i$-th and the $i+1$-st strands, is represented in $\Gamma_{0,(n),1}$ by a half Dehn twist $h_i$ along a simple closed curve surrounding the $i$-th and the $i+1$-st holes in the disk. Now consider two identical surfaces of $S_{0,n+1}$ which are mirror symmetric (See Fig.~\ref{mirror-surface}).

\begin{figure}[ht]
  \includegraphics[width=0.55\textwidth]{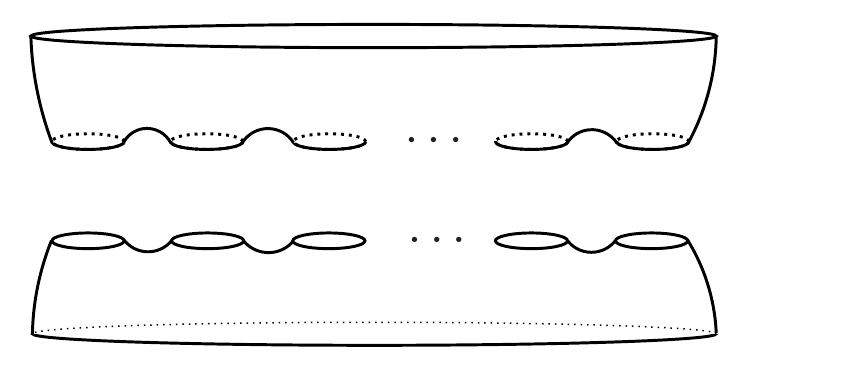}\\
\caption{}\label{mirror-surface}
\end{figure}

By gluing two symmetric surfaces of $S_{0,n+1}$ along $n$ holes, we get a surface $S_{n-1,2}$. A regular neighborhood of the circle obtained by gluing two holes is called a {\it pillar.} See Fig.~\ref{Pillars}.

\begin{figure}[ht]
  \includegraphics[width=0.55\textwidth]{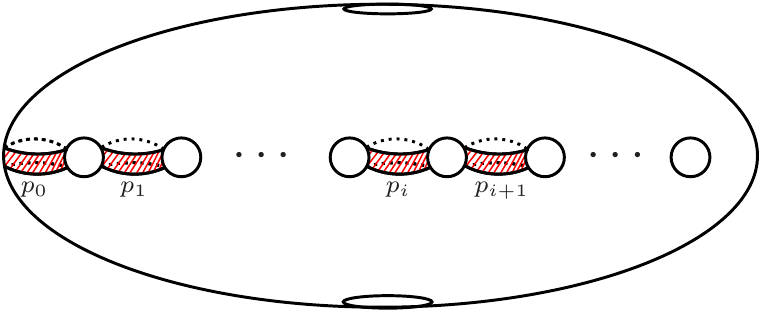}\\
\caption{Pillars $p_0,\ldots,p_g$ of surface $S_{g,2}$}\label{Pillars}
\end{figure}

The half Dehn twists in $\Gamma_{0,(g),1}$ can be extended to self-diffeomorphisms on the surface $S_{g-1, 2}$ which is obtained by gluing two mirror symmetric surfaces of $S_{0, g+1}$. Those self-diffeomorphisms are called  {\it pillar switchings}, denoted by $\sigma_0,\ldots,\sigma_{g-1}$ in $\Gamma_{g-1,2}$, that is,  $\sigma_i$ is a horizontal half twist switching two pillars $p_i$ and $p_{i+1}$. The obvious embedding $B_g \hookrightarrow \Gamma_{0,(g),1}$ is extended to $$\phi : B_g \rightarrow \Gamma_{g-1, 2} \quad \beta_i \mapsto \sigma_{i-1}$$
which is obviously well-defined.

Let $x_1,y_1,\ldots,x_g,y_g$ be the generators of the fundamental group of $S_{g,1}$ which are parallel to
the Dehn twists $a_1,b_1,\ldots,a_g,b_g$, respectively. The mapping class group $\Gamma_{g,1}$ is identified with the subgroup of the automorphism group of the free group on $x_1,y_1,\ldots,x_g,y_g$ that consists of the automorphisms fixing the fundamental relator $R=[y_1,x_1]\cdots[y_g,x_g]$ which represents a loop along the boundary of $S_{g,1}$. That is, each element is uniquely determined by the action on $\pi_1 S_{g,1}$. In Section 2, we give a complete list of the actions of pillar switchings on $\pi_1 S_{g,1}$ (Theorem 2.1). In Theorem 2.2 we give complete expressions of pillar switchings in terms of the standard Dehn twists $a_1, \ldots , a_g, b_1 , \ldots, b_g, w_1, \ldots, w_{g-1}$ (Fig.~\ref{Dehn twists}).

Let $$\phi' : B_g \rightarrow \Gamma_{g, 2} \quad \beta_i \mapsto \sigma_{i}.$$
This map omits $\sigma_0$ in the range. We will deal with this map rather than \\ $\phi : B_g \rightarrow \Gamma_{g-1, 2}$  for a technical reason. Let $\psi:B_g\rightarrow\Gamma_{g,1}$ be the composite
 $$B_g \stackrel{\phi'}{\rightarrow} \Gamma_{g,2} \stackrel{I}{\rightarrow}\Gamma_{g,1},$$
where the map $I: \Gamma_{g,2} \rightarrow \Gamma_{g,1} $ is induced by attaching a disk to a boundary component of $S_{g,2}$. We show (Theorem 4.1) that the homomorphism $\psi : B_g \rightarrow \Gamma_{g,1}$ is injective.

We also have the following:

\begin{Theorem} The homomorphism $\psi_*  : H_*(B_{\infty};\mathbb Z)\rightarrow H_*(\Gamma_{\infty,1};\mathbb Z)$ induced by $\psi$ is trivial.
\end{Theorem}

 The strategy of proof of Theorem 4.2 is categorical delooping, which is analogous to \cite{Song-Tillmann}. However, in this paper we need some hard work in constructing suitable functor from (suitably modified) tile category to surface category.

 The category $\mathcal B=\coprod_{n\ge 0} B_n$ which is a disjoint union of braid groups is a typical braided monoidal category. The group completion of its classifying space $\coprod_{n\ge 0} BB_n$ is homotopy equivalent to $\mathbb Z\times BB_{\infty}^+$. It is well-known that $BB_{\infty}^+ \simeq\Omega_0^2 S^2 \simeq\Omega^2 S^3$. Cohen (\cite{CohenLM}) proved that
$$H_*(B_{\infty};\mathbb Z/2)\simeq \mathbb Z/2 [x_1,x_2,\ldots]$$
as polynomial algebras, where $\deg x_i=2^i-1$ and $Q_1(x_i)=x_{i+1}$, where $Q_1:H_k(X;\mathbb Z/2)\rightarrow H_{2k-1}(X;\mathbb Z/2)$ is the first Araki-Kudo-Dyer-Lashof operation. If we show that $\psi_*: H_*(B_{\infty};\mathbb Z/2)$ $\rightarrow H_*(\Gamma_{\infty,1};\mathbb Z/2)$ preserves $Q_1$, then $\psi_*=0$ because $H_1(\Gamma_{\infty,1};\mathbb Z/2)=0$ ($\Gamma_{g,1}$ is perfect for $g\ge 3$). Since $Q_1$ arises from double loop space structure, if $\psi:BB_{\infty}^+\rightarrow B\Gamma_{\infty,1}^+$ is a double loop space map, then
$$\psi_*: H_*(B_{\infty};\mathbb Z/2)\rightarrow H_*(\Gamma_{\infty,1};\mathbb Z/2)$$
is trivial.

In order to show that $\psi:BB_{\infty}^+\rightarrow B\Gamma_{\infty,1}^+$ is a double loop space map, we lift it to a functor $\Psi:\mathcal T \rightarrow \mathcal S$ between two monoidal 2-categories. $\mathcal T$, called tile category, is a categorical model for braid groups. There is, in this paper, a modification of the tile category given in \cite{Song-Tillmann}.  Surface category $\mathcal S$ was constructed in \cite{Tillmann2,Tillmann3}. In this paper we also make some modifications on $\mathcal S$ in order to construct a desired 2-functor $\Psi$. The construction of a suitable functor is one of the key parts of this paper.


\section{The pillar switchings}\label{sec:half-twist}

In this section we give a complete list of actions of pillar switchings on $\pi_1 S_{g,1}$. we also give algebraic interpretations of pillar switchings in terms of the standard Dehn twists. For this we need to calculate chains of consecutive actions of Dehn twists on the fundamental group of $S_{g,1}$.

Let $\Gamma_{g,1}$ be the mapping class group of orientation preserving self-homeomorph\-isms of $S_{g,1}$ which fix the  boundary pointwise. Then $\Gamma_{g,1}$ may be regarded as a subgroup of Aut{$\pi_1 S_{g,1}$ which consists of automorphisms fixing the loop along the boundary.
\begin{figure}[ht]
  \includegraphics[width=0.55\textwidth]{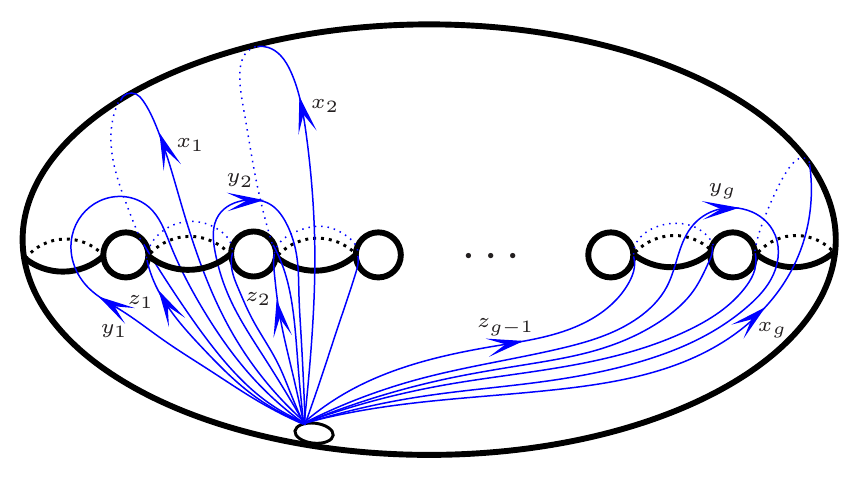}\\
  \caption{The generators $x_1,y_1,\ldots,x_g,y_g$ of $\pi_1 S_{g,1}$}\label{Generating-Loop}
\end{figure}

We first need to understand the actions of the standard Dehn twists on the fundamental group of surface $S_{g,1}$. Let $x_1,y_1,\ldots,x_g,y_g$ be generators of the fundamental group of $S_{g,1}$ which are parallel to
the Dehn twists $a_1,b_1,\ldots,a_g,b_g$, respectively. The mapping class group $\Gamma_{g,1}$ is identified with the subgroup of the automorphism group of the free group on $x_1,y_1,\ldots,x_g,y_g$ that consists of the automorphisms fixing the fundamental relator $R=[y_1,x_1]\cdots[y_g,x_g]$. The following is a well-known result.

\begin{lemma}[The actions of Dehn twists] Let $a_1,\ldots,a_g,b_1,\ldots,b_g,w_1,\ldots,w_{g-1}$ be the Dehn twists as in Fig.~\ref{Dehn twists} which generate $\Gamma_{g,1}$. Then they act on $\pi_1 S_{g,1}$ as follows$:$
$$
\begin{aligned}
a_i:
y_i&\mapsto y_ix_i^{-1}\\
b_i: x_i&\mapsto x_iy_i\\
w_i:
x_i&\mapsto z_i^{-1}y_{i+1}x_{i+1}y_{i+1}^{-1}\\
y_i&\mapsto y_iz_i\\
y_{i+1}&\mapsto z_i^{-1}y_{i+1}
\end{aligned} 
$$
where $z_i=x_i^{-1}y_{i+1}x_{i+1}y_{i+1}^{-1}$ and these automorphisms fix the generators that do not appear in the above list.
\end{lemma}

Note that $z_i$ represents a loop parallel to the Dehn twist $w_i$.


As an element of $\Gamma_{g,1}$ each pillar switching acts on the fundamental group of $S_{g,1}$ which is a free group on $y_1, x_1, \ldots , y_g, x_g$:

\begin{theorem} Each pillar switching $\sigma_i$ acts on $\pi_1 S_{g,1}$ as follows$:$
\begin{align*}
{\sigma}_0: x_1 &\longmapsto z_1^{-1}y_1z_1y_1^{-1}z_1\\
y_1&\longmapsto z_1^{-1}y_1^{-1}z_1\\
y_2&\longmapsto z_1^{-1}y_1z_1y_2,
\\
{\sigma}_{i-1}:  x_{i-1}&\longmapsto y_i^{-1}x_{i-1}y_i\\
x_i&\longmapsto z_i^{-1}y_iz_iy_i^{-1}x_{i-1}z_i\\
y_{i-1}&\longmapsto y_{i-1}y_i\\
y_i&\longmapsto z_i^{-1}y_i^{-1}z_i\\
y_{i+1}&\longmapsto z_i^{-1}y_ix_i^{-1}y_{i+1}x_{i+1}=z_i^{-1}y_iz_iy_{i+1},
\\
{\sigma}_{g-1}: x_{g-1}&\longmapsto y_g^{-1}x_{g-1}y_g\\
x_g&\longmapsto x_gz_{g-1}^{-1}x_g^{-1}\\
y_{g-1}&\longmapsto y_{g-1}y_g\\
y_g&\longmapsto x_gy_g^{-1}x_g^{-1}.
\end{align*}
Here, $2\le i \le g-1$ and these automorphisms fix the generators that do not appear in the list.
\end{theorem}

\begin{proof} Fig.~\ref{fig5}-\ref{fig16} prove this theorem.

\begin{figure}[h!]
  \includegraphics[width= 0.75\textwidth]{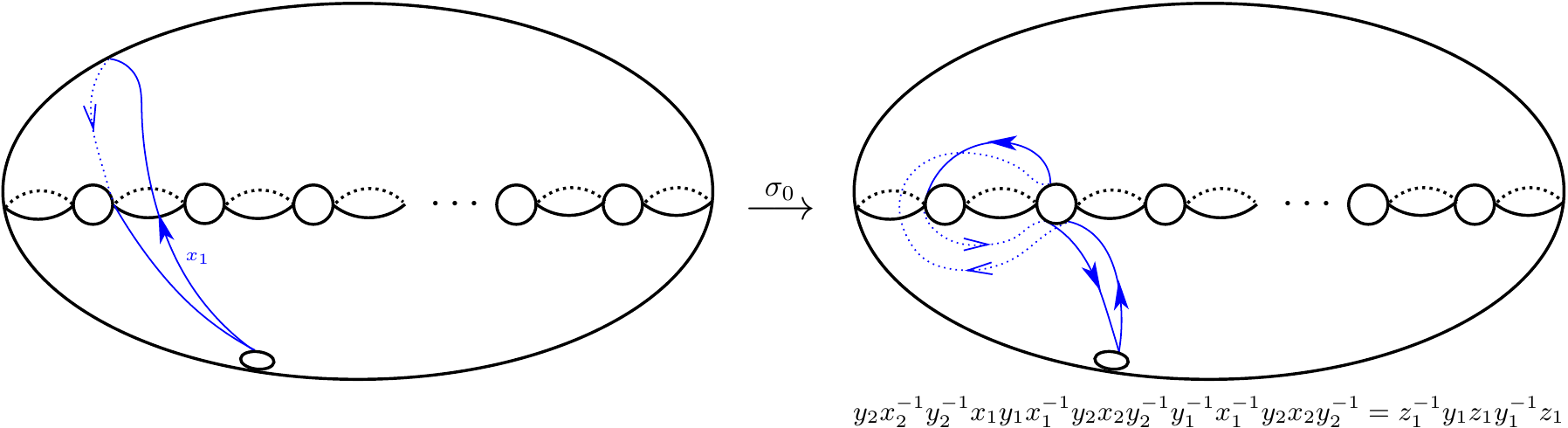}\\
  \caption{The action of ${\sigma}_0$ on $x_1$}\label{fig5}
\end{figure}
\begin{figure}[h!]
  \includegraphics[width= 0.75\textwidth]{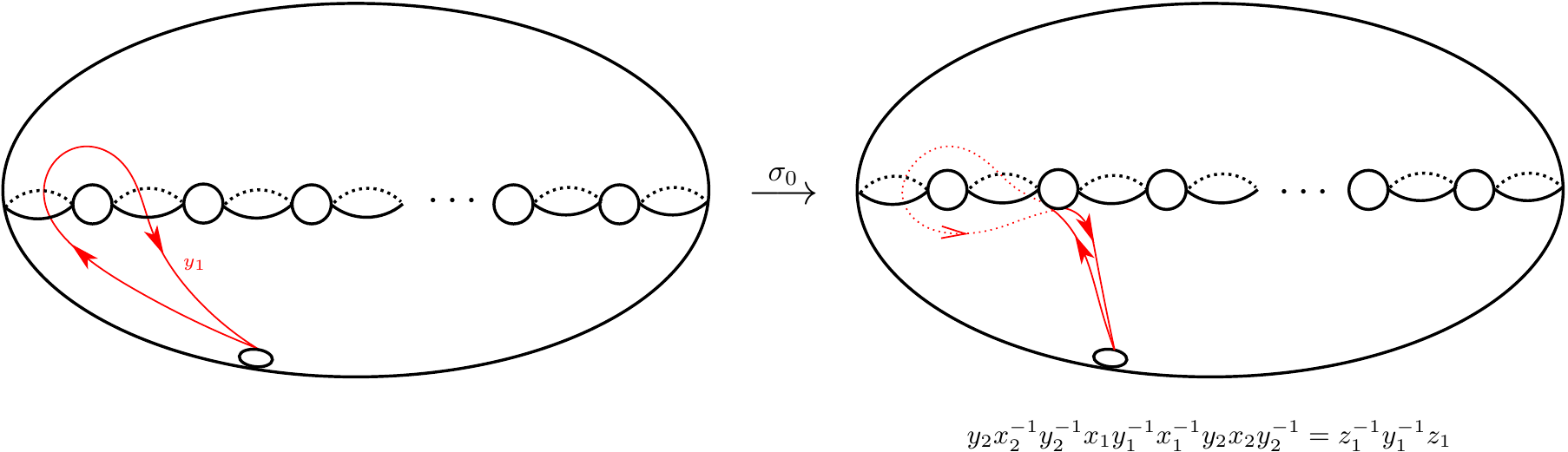}\\
  \caption{The action of ${\sigma}_0$ on $y_1$}\label{fig6}
\end{figure}
\begin{figure}[h!]
  \includegraphics[width= 0.75\textwidth]{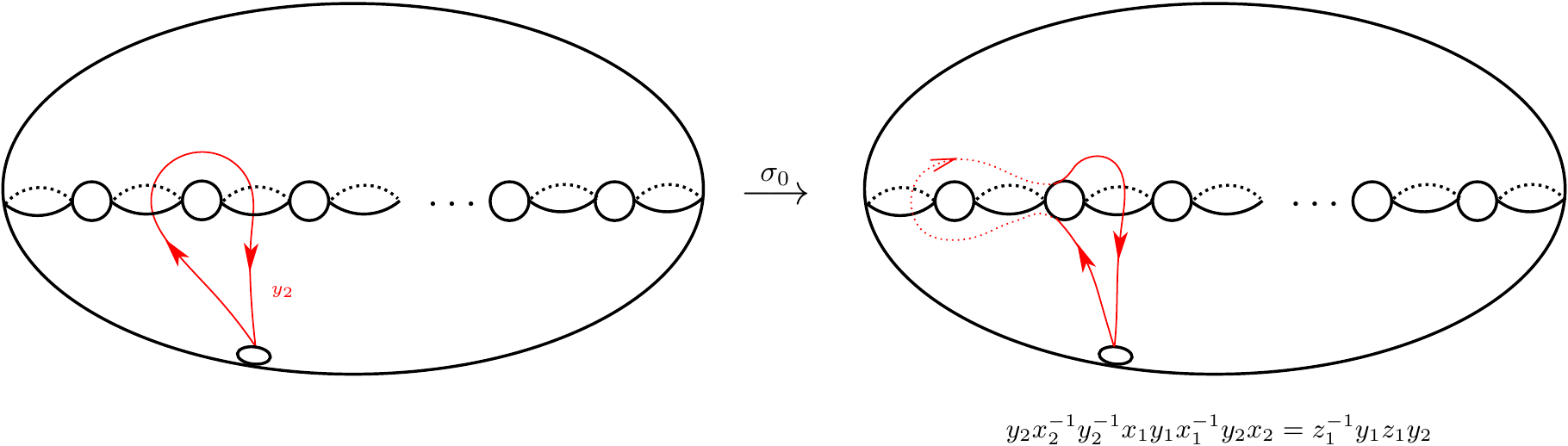}\\
  \caption{The action of ${\sigma}_0$ on $y_{2}$}\label{fig7}
\end{figure}
\begin{figure}[h!]
  \includegraphics[width= 0.75\textwidth]{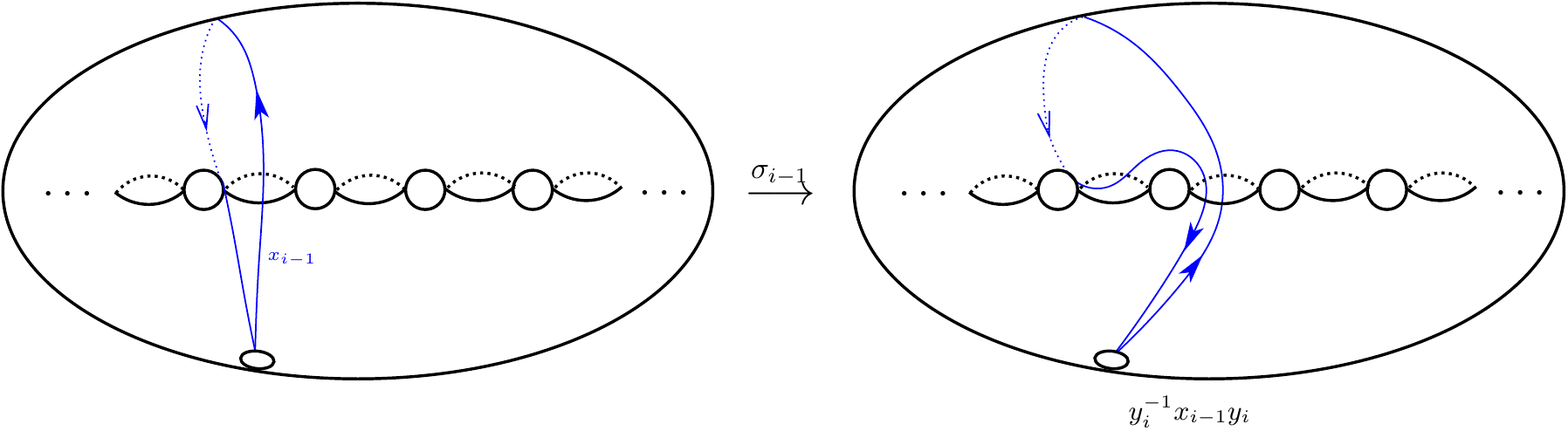}\\
  \caption{The action of ${\sigma}_{i-1}$ on $x_{i-1}$}\label{fig8}
\end{figure}
\begin{figure}[h!]
  \includegraphics[width= 0.75\textwidth]{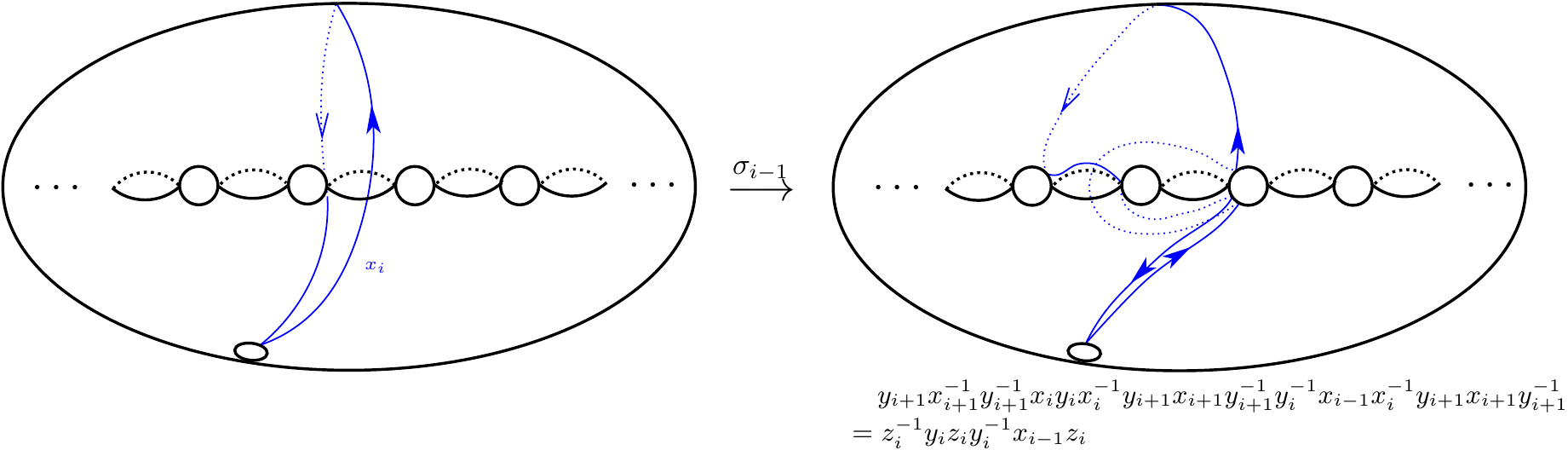}\\
  \caption{The action of ${\sigma}_{i-1}$ on $x_i$}\label{fig9}
\end{figure}
\begin{figure}[h!]
  \includegraphics[width= 0.75\textwidth]{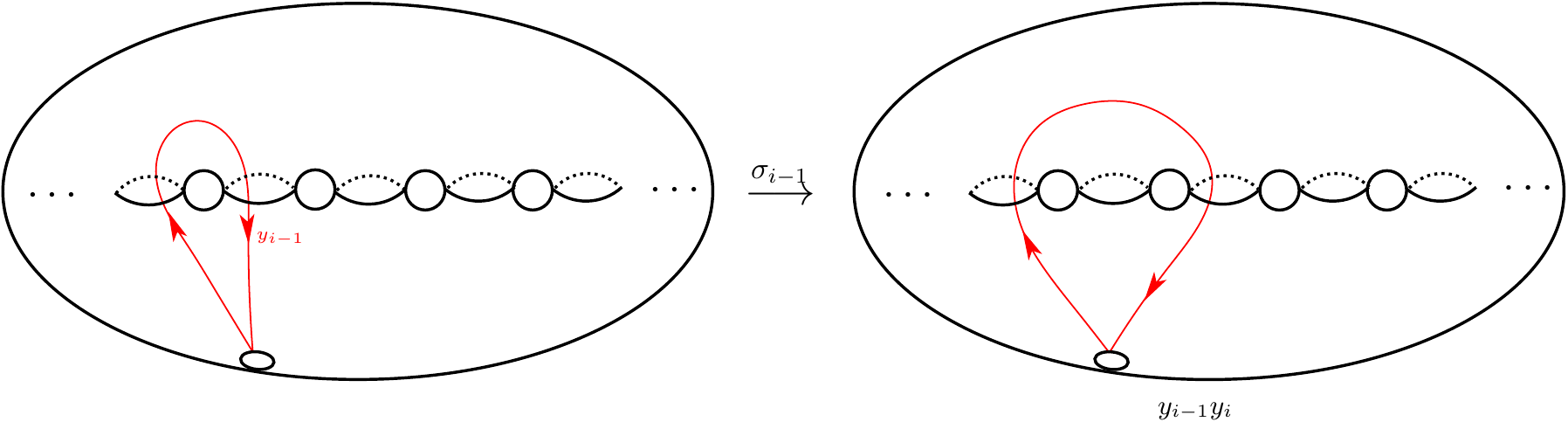}\\
  \caption{The action of ${\sigma}_{i-1}$ on $y_{i-1}$}\label{fig10}
\end{figure}
\begin{figure}[h!]
  \includegraphics[width= 0.75\textwidth]{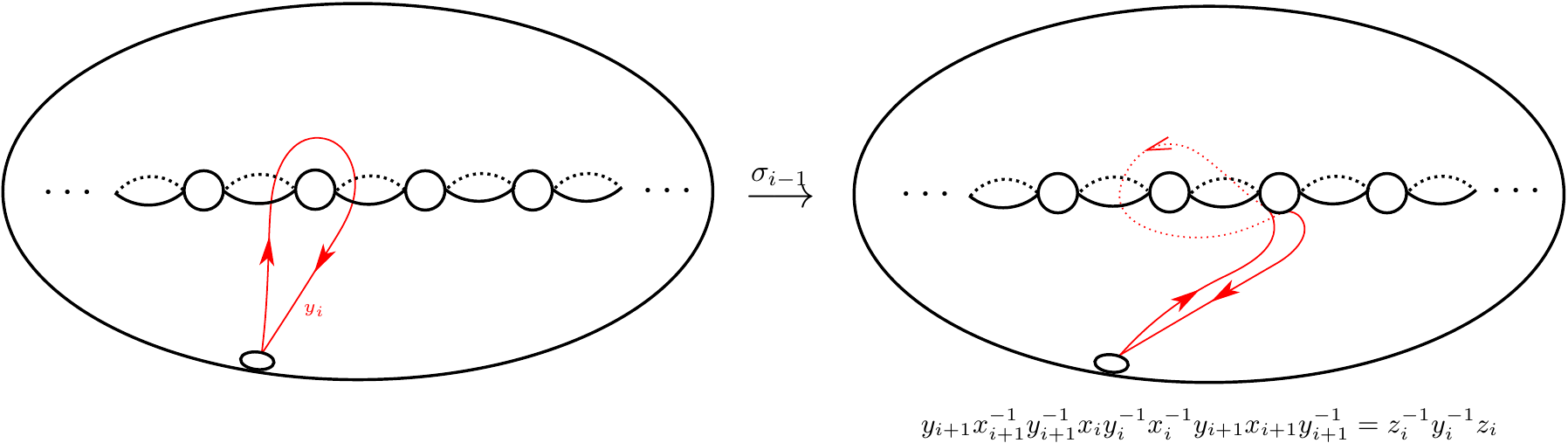}\\
  \caption{The action of ${\sigma}_{i-1}$ on $y_{i}$}\label{fig11}
\end{figure}
\begin{figure}[h!]
  \includegraphics[width= 0.75\textwidth]{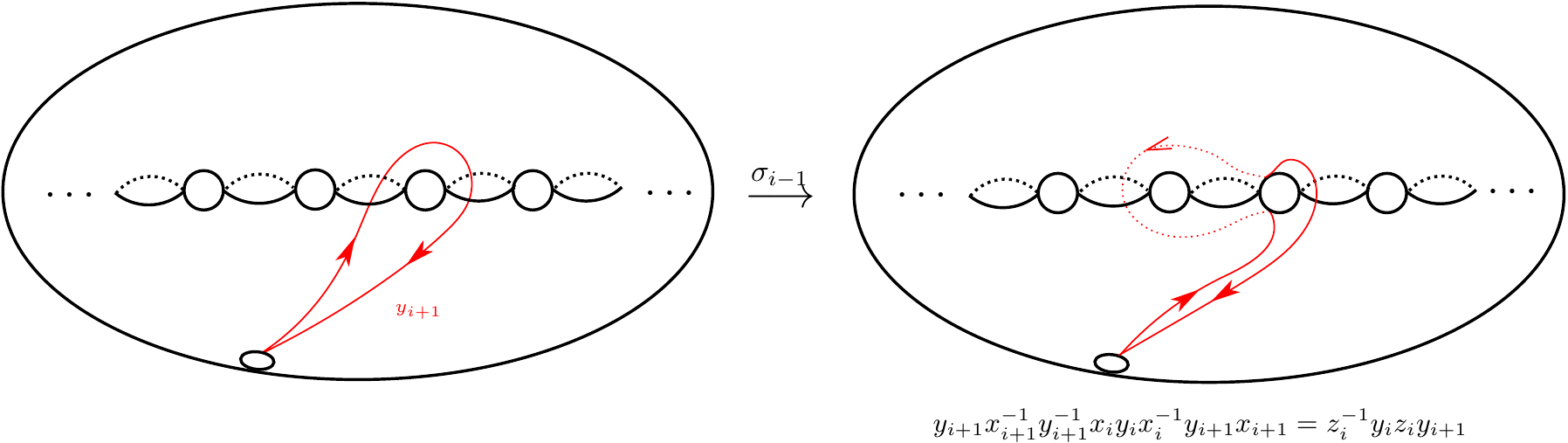}\\
  \caption{The action of ${\sigma}_{i-1}$ on $y_{i+1}$}\label{fig12}
\end{figure}
\begin{figure}[h!]
  \includegraphics[width= 0.75\textwidth]{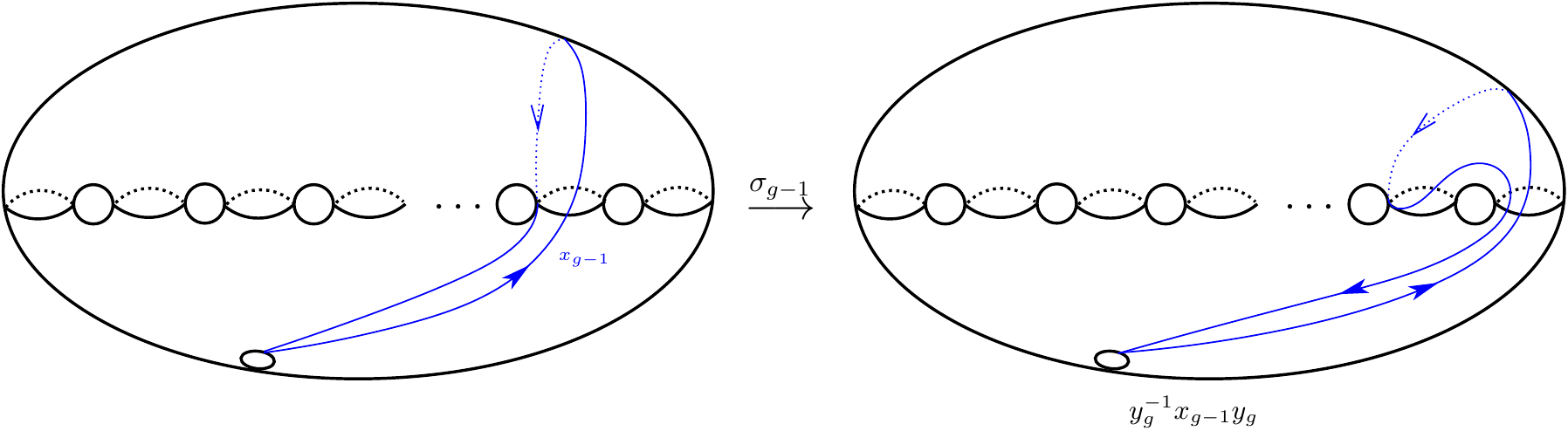}\\
  \caption{The action of ${\sigma}_{g}$ on $x_{g-1}$}\label{fig13}
\end{figure}
\begin{figure}[h!]
  \includegraphics[width= 0.75\textwidth]{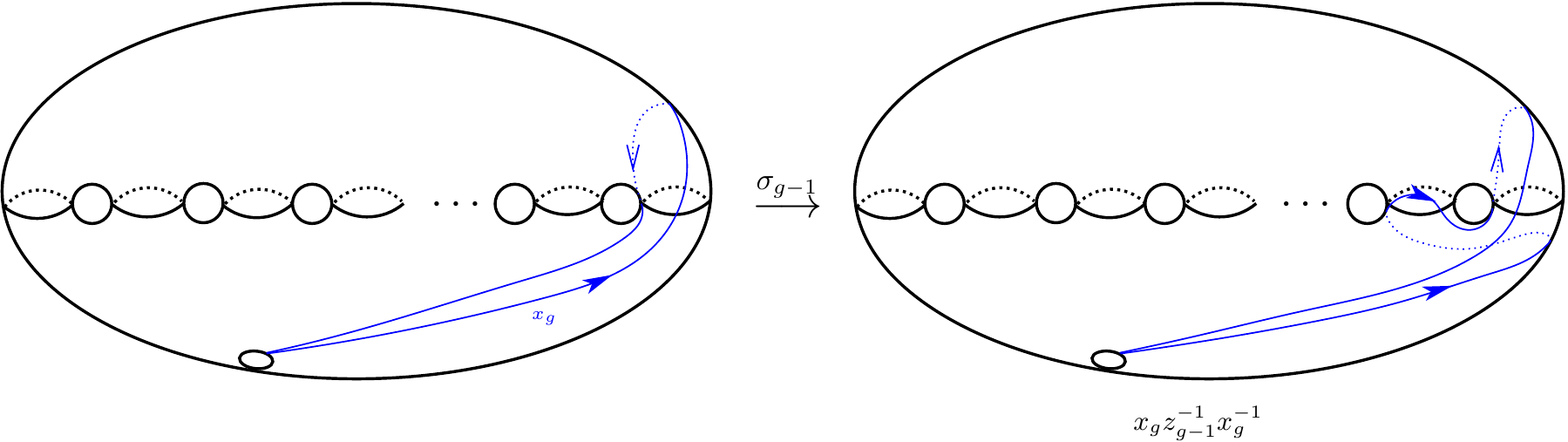}\\
  \caption{The action of ${\sigma}_{g-1}$ on $x_{g}$}\label{fig14}
\end{figure}
\begin{figure}[h!]
  \includegraphics[width= 0.75\textwidth]{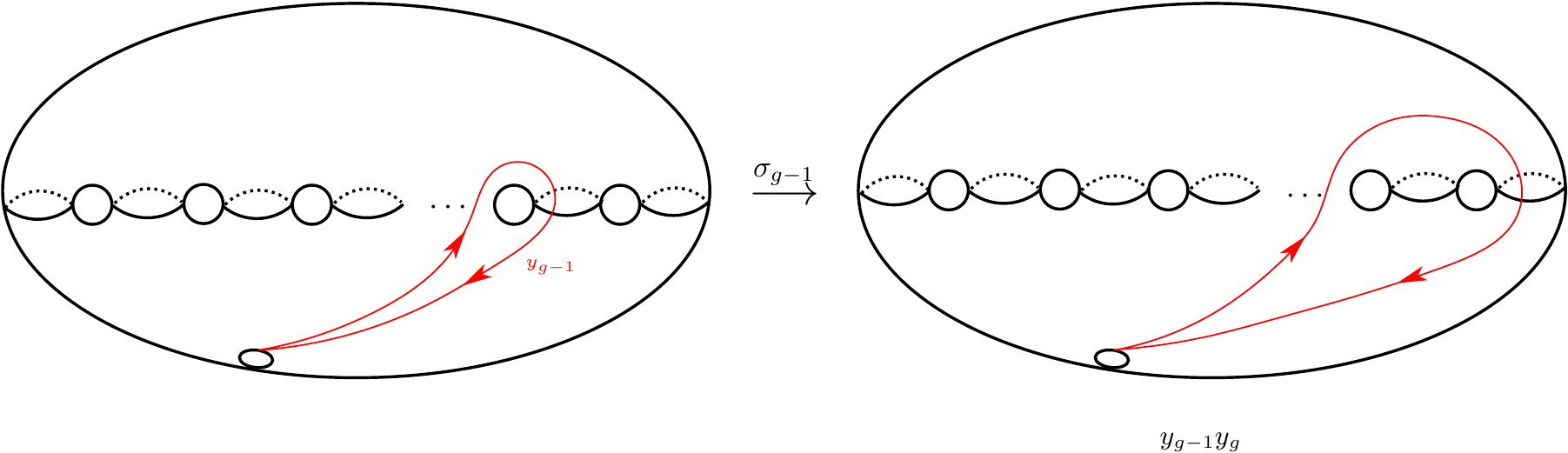}\\
  \caption{The action of ${\sigma}_{g-1}$ on $y_{g-1}$}\label{fig15}
\end{figure}
\begin{figure}[h!]
  \includegraphics[width= 0.75\textwidth]{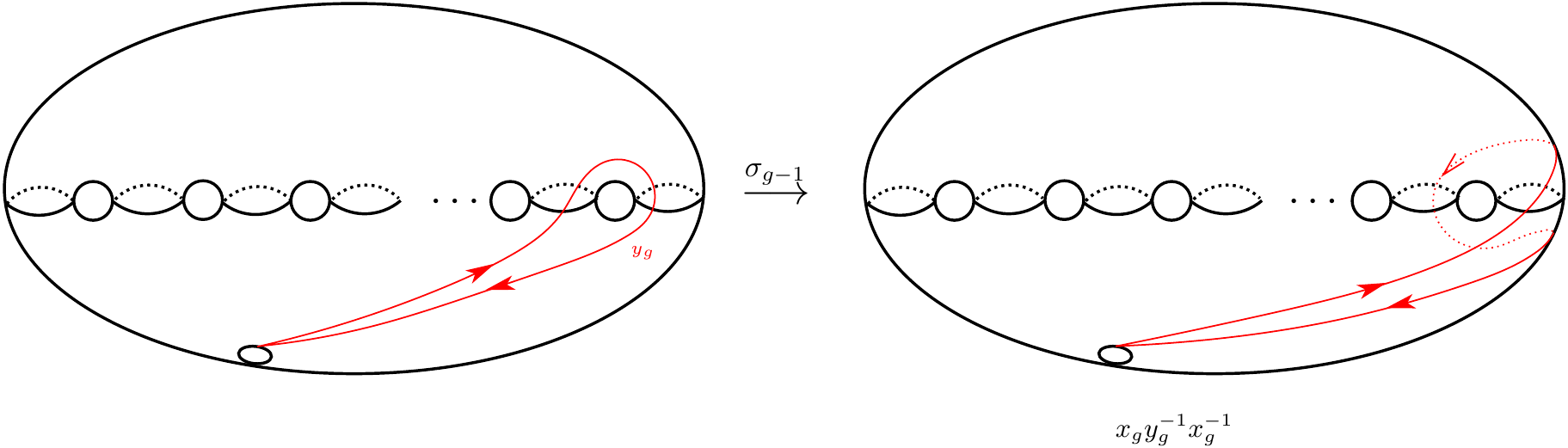}\\
  \caption{The action of ${\sigma}_{g-1}$ on $y_{g}$}\label{fig16}
\end{figure}
\end{proof}

 From Theorem 2.1 we immediately get the following result which will play a key role in the proof of Theorem 4.1.

\begin{corollary} In $\pi_1 S_{g,1}$, let $z_i = x_i^{-1}y_{i+1}x_{i+1}y_{i+1}^{-1}$ for $1 \le i \le g-1$, and let $z_g=x_g^{-1}$. Then $\pi_1 S_{g,1}$ is a free group on generators $y_1, y_2, \ldots, y_g, z_1, z_2,\ldots, z_g$, and each pillar switching $\sigma_i$ acts on $\pi_1 S_{g,1}=F_{\{y_1, y_2, \ldots, y_g, z_1, z_2,\ldots, z_g\}}$ as follows$:$
\begin{align*}
{\sigma}_{i-1}:
y_{i-1}&\longmapsto y_{i-1}y_i\\
y_i&\longmapsto z_i^{-1}y_i^{-1}z_i\\
y_{i+1}&\longmapsto z_i^{-1}y_iz_iy_{i+1}\\
z_{i-1}&\longmapsto z_i\\
z_i&\longmapsto z_i^{-1}z_{i-1}z_i,
\\
{\sigma}_{g-1}:
y_{g-1}&\longmapsto y_{g-1}y_g\\
y_g&\longmapsto z_g^{-1}y_g^{-1}z_g\\
z_{g-1}&\longmapsto z_g\\
z_g&\longmapsto z_g^{-1}z_{g-1}z_g.
\end{align*}
Here, $2\le i \le g-1$ and these automorphisms fix the generators that do not appear in the list.
\end{corollary}

As elements of $\Gamma_{g,1}$, pillar switchings should be able to be expressed in terms of products of Dehn twists. We now give these expressions which is of interest in its own right.

\begin{theorem} We have the following equalities$:$
\begin{itemize}
\item[$(1)$] ${\sigma}_0=a_2^{-1}w_1a_1b_1w_1a_1b_1=a_2^{-1}(w_1a_1b_1)^2;$
\item[$(2)$] ${\sigma}_{i-1}=a_{i+1}^{-1}a_ib_iw_iw_{i-1}a_{i-1}^{-1}b_ia_i$ \ \ for $2 \le i \le g-1;$
\item[$(3)$] ${\sigma}_{g-1}=w_{g-1}a_gb_gw_{g-1}a_gb_ga_{g-1}^{-1}=(w_{g-1}a_gb_g)^2a_{g-1}^{-1}$.
\end{itemize}
\end{theorem}
\begin{proof}

Each element in $\Gamma_{g,1}$ is uniquely determined by the action on $\pi_1 S_{g,1}$ which is a free group on $y_1, x_1, \ldots , y_g, x_g$. We  will prove the equalities of this theorem by comparing the actions of the right-hand side of the equalities on $\pi_1 S_{g,1}$ with those given in Theorem 2.1.

\begin{align*}
(1)\\
x_1&\stackrel{b_1}{\longmapsto}x_1y_1 \stackrel{a_1}{\longmapsto}x_1y_1x_1^{-1} \stackrel{w_1}{\longmapsto}z_1^{-1}y_2x_2y_2^{-1}y_1z_1y_2x_2^{-1}y_2^{-1}z_1=z_1^{-1}y_2x_2y_2^{-1}y_1x_1^{-1}z_1\\
&\stackrel{b_1}{\longmapsto}z_1^{-1}y_1y_2x_2y_2^{-1}x_1^{-1}y_1^{-1}z_1 \stackrel{a_1}{\longmapsto}z_1^{-1}y_1z_1y_1^{-1}z_1 \stackrel{w_1}{\longmapsto}z_1^{-1}y_1z_1y_1^{-1}z_1\\& \stackrel{a_2^{-1}}{\longmapsto}z_1^{-1}y_1z_1y_1^{-1}z_1;\\
y_1&\stackrel{b_1}{\longmapsto}y_1 \stackrel{a_1}{\longmapsto}y_1x_1^{-1} \stackrel{w_1}{\longmapsto}y_1z_1y_2x_2^{-1}y_2^{-1}z_1 \stackrel{b_1}{\longmapsto}z_1y_2x_2^{-1}y_2^{-1}y_1^{-1}z_1 \stackrel{a_1}{\longmapsto}y_1^{-1}z_1 \\
&\stackrel{w_1}{\longmapsto}z_1^{-1}y_1^{-1}z_1 \stackrel{a_2^{-1}}{\longmapsto}z_1^{-1}y_1^{-1}z_1;\\
y_2&\stackrel{b_1}{\longmapsto}y_2 \stackrel{a_1}{\longmapsto}y_2 \stackrel{w_1}{\longmapsto}z_1^{-1}y_2 \stackrel{b_1}{\longmapsto}z_1^{-1}y_1y_2 \stackrel{a_1}{\longmapsto}z_1^{-1}y_1x_1^{-1}y_2\\
&\stackrel{w_1}{\longmapsto}z_1^{-1}y_1z_1y_2x_2^{-1}y_2^{-1}z_1z_1^{-1}y_2=z_1^{-1}y_1z_1y_2x_2^{-1}\\ &\stackrel{a_2^{-1}}{\longmapsto}z_1^{-1}y_1z_1y_2x_2x_2^{-1}=z_1^{-1}y_1z_1y_2;\\
z_1&\stackrel{b_1}{\longmapsto}y_1^{-1}z_1\stackrel{a_1}{\longmapsto}x_1y_1^{-1}z_1\stackrel{w_1}{\longmapsto}z_1^{-1}y_2x_2y_2^{-1}z_1^{-1}y_1^{-1}z_1\\
&\stackrel{b_1}{\longmapsto}z_1^{-1}y_1y_2x_2y_2^{-1}z_1^{-1}y_1y_1^{-1}y_1^{-1}z_1=z_1^{-1}y_1y_2x_2y_2^{-1}z_1^{-1}y_1^{-1}z_1\\
&\stackrel{a_1}{\longmapsto}z_1^{-1}y_1x_1^{-1}y_2x_2y_2^{-1}z_1^{-1}x_1y_1^{-1}z_1\\
&\stackrel{w_1}{\longmapsto}z_1^{-1}y_1z_1y_2x_2^{-1}y_2^{-1}z_1z_1^{-1}y_2x_2y_2^{-1}z_1z_1^{-1}z_1^{-1}y_2x_2y_2^{-1}z_1^{-1}y_1^{-1}z_1\\
&\qquad=z_1^{-1}y_1y_2x_2y_2^{-1}z_1^{-1}y_1^{-1}z_1\\
&\stackrel{a_2^{-1}}{\longmapsto}z_1^{-1}y_1x_1y_1^{-1}z_1;
\end{align*}

\begin{align*}
(2)\\
x_{i-1}&\stackrel{a_i}{\longmapsto}x_{i-1}\stackrel{b_i}{\longmapsto}x_{i-1}\stackrel{a_{i-1}^{-1}}{\longmapsto}x_{i-1}\stackrel{w_{i-1}}{\longmapsto}z_{i-1}^{-1}y_ix_iy_i^{-1}\\
&\stackrel{w_i}{\longmapsto}z_{i-1}^{-1}y_iz_iz_i^{-1}y_{i+1}x_{i+1}y_{i+1}^{-1}z_i^{-1}y_i^{-1}=z_{i-1}^{-1}y_iy_{i+1}x_{i+1}y_{i+1}^{-1}z_i^{-1}y_i^{-1}\\
&\stackrel{b_i}{\longmapsto}y_i^{-1}z_{i-1}^{-1}y_iy_{i+1}x_{i+1}y_{i+1}^{-1}z_i^{-1}y_iy_i^{-1}=y_i^{-1}z_{i-1}^{-1}y_iy_{i+1}x_{i+1}y_{i+1}^{-1}z_i^{-1}\\
&\stackrel{a_i}{\longmapsto}x_iy_i^{-1}z_{i-1}^{-1}y_ix_i^{-1}y_{i+1}x_{i+1}y_{i+1}^{-1}z_i^{-1}=y_i^{-1}x_{i-1}y_i\\
&\stackrel{a_{i+1}^{-1}}{\longmapsto}y_i^{-1}x_{i-1}y_i;\\
x_i&\stackrel{a_i}{\longmapsto}x_i\stackrel{b_i}{\longmapsto}x_iy_i\stackrel{a_{i-1}^{-1}}{\longmapsto}x_iy_i\stackrel{w_{i-1}}{\longmapsto}x_iz_{i-1}^{-1}y_i\stackrel{w_i}{\longmapsto}z_i^{-1}y_{i+1}x_{i+1}y_{i+1}^{-1}z_{i-1}^{-1}y_iz_i\\
&\stackrel{b_i}{\longmapsto}z_i^{-1}y_iy_{i+1}x_{i+1}y_{i+1}^{-1}y_i^{-1}z_{i-1}^{-1}y_iy_i^{-1}z_i=z_i^{-1}y_iy_{i+1}x_{i+1}y_{i+1}^{-1}y_i^{-1}z_{i-1}^{-1}z_i\\
&\stackrel{a_i}{\longmapsto}z_i^{-1}y_ix_i^{-1}y_{i+1}x_{i+1}y_{i+1}^{-1}x_iy_i^{-1}z_{i-1}^{-1}z_i=z_i^{-1}y_iz_iy_i^{-1}x_{i-1}z_i\\
&\stackrel{a_{i+1}^{-1}}{\longmapsto}z_i^{-1}y_iz_iy_i^{-1}x_{i-1}z_i;\\
y_{i-1}&\stackrel{a_i}{\longmapsto}y_{i-1}\stackrel{b_i}{\longmapsto}y_{i-1}\stackrel{a_{i-1}^{-1}}{\longmapsto}y_{i-1}x_{i-1}\stackrel{w_{i-1}}{\longmapsto}y_{i-1}z_{i-1}z_{i-1}^{-1}y_ix_iy_i^{-1}=y_{i-1}y_ix_iy_i^{-1}\\
&\stackrel{w_i}{\longmapsto}y_{i-1}y_iz_iz_i^{-1}y_{i+1}x_{i+1}y_{i+1}^{-1}z_i^{-1}y_i^{-1}=y_{i-1}y_iy_{i+1}x_{i+1}y_{i+1}^{-1}z_i^{-1}y_i^{-1}\\
&\stackrel{b_i}{\longmapsto}y_{i-1}y_iy_{i+1}x_{i+1}y_{i+1}^{-1}z_i^{-1}y_iy_i^{-1}=y_{i-1}y_iy_{i+1}x_{i+1}y_{i+1}^{-1}z_i^{-1}\\
&\stackrel{a_i}{\longmapsto}y_{i-1}y_ix_i^{-1}y_{i+1}x_{i+1}y_{i+1}^{-1}z_i^{-1}=y_{i-1}y_i\stackrel{a_{i+1}^{-1}}{\longmapsto}y_{i-1}y_i;\\
y_i&\stackrel{a_i}{\longmapsto}y_ix_i^{-1}\stackrel{b_i}{\longmapsto}y_iy_i^{-1}x_i^{-1}=x_i^{-1}\stackrel{a_{i-1}^{-1}}{\longmapsto}x_i^{-1}\stackrel{w_{i-1}}{\longmapsto}x_i^{-1}\\
&\stackrel{w_i}{\longmapsto}y_{i+1}x_{i+1}^{-1}y_{i+1}^{-1}z_i\stackrel{b_i}{\longmapsto}y_{i+1}x_{i+1}^{-1}y_{i+1}^{-1}y_i^{-1}z_i\\
&\stackrel{a_i}{\longmapsto}y_{i+1}x_{i+1}^{-1}y_{i+1}^{-1}x_iy_i^{-1}z_i=z_i^{-1}y_i^{-1}z_i\\
&\stackrel{a_{i+1}^{-1}}{\longmapsto}z_i^{-1}y_i^{-1}z_i;\\
y_{i+1}&\stackrel{a_i}{\longmapsto}y_{i+1}\stackrel{b_i}{\longmapsto}y_{i+1}\stackrel{a_{i-1}^{-1}}{\longmapsto}y_{i+1}\stackrel{w_{i-1}}{\longmapsto}y_{i+1}\stackrel{w_i}{\longmapsto}z_i^{-1}y_{i+1}\stackrel{b_i}{\longmapsto}z_i^{-1}y_iy_{i+1}\\
&\stackrel{a_i}{\longmapsto}z_i^{-1}y_ix_i^{-1}y_{i+1}\stackrel{a_{i+1}^{-1}}{\longmapsto}z_i^{-1}y_ix_i^{-1}y_{i+1}x_{i+1}=z_i^{-1}y_iz_iy_{i+1};\\
z_{i-1}&\stackrel{a_i}{\longmapsto}z_{i-1}\stackrel{b_i}{\longmapsto}z_{i-1}y_i\stackrel{a_{i-1}^{-1}}{\longmapsto}z_{i-1}y_i\stackrel{w_{i-1}}{\longmapsto}z_{i-1}z_{i-1}^{-1}y_i=y_i\stackrel{w_i}{\longmapsto}y_iz_i\stackrel{b_i}{\longmapsto}y_iy_i^{-1}z_i=z_i\\
&\stackrel{a_i}{\longmapsto}z_i\stackrel{a_{i+1}^{-1}}{\longmapsto}z_i;\\
z_i&\stackrel{a_i}{\longmapsto}z_i\stackrel{b_i}{\longmapsto}y_i^{-1}z_i\stackrel{a_{i-1}^{-1}}{\longmapsto}y_i^{-1}z_i\stackrel{w_{i-1}}{\longmapsto}y_i^{-1}z_{i-1}z_i\stackrel{w_i}{\longmapsto}z_i^{-1}y_i^{-1}z_{i-1}z_i\\
&\stackrel{b_i}{\longmapsto}z_i^{-1}y_iy_i^{-1}z_{i-1}y_iy_i^{-1}z_i=z_i^{-1}z_{i-1}z_i\stackrel{a_i}{\longmapsto}z_i^{-1}z_{i-1}z_i\stackrel{a_{i+1}^{-1}}{\longmapsto}z_i^{-1}z_{i-1}z_i;
\end{align*}

\begin{align*}
(3)\\
x_{g-1}&\stackrel{a_{g-1}^{-1}}{\longmapsto}x_{g-1}\stackrel{b_g}{\longmapsto}x_{g-1}\stackrel{a_g}{\longmapsto}x_{g-1}\stackrel{w_{g-1}}{\longmapsto}z_{g-1}^{-1}y_gx_gy_g^{-1}\\
&\stackrel{b_g}{\longmapsto}y_g^{-1}z_{g-1}^{-1}y_gx_gy_gy_g^{-1}=y_g^{-1}z_{g-1}^{-1}y_gx_g\\
&\stackrel{a_g}{\longmapsto}x_gy_g^{-1}z_{g-1}^{-1}y_gx_g^{-1}x_g=y_g^{-1}x_{g-1}y_g\\
&\stackrel{w_{g-1}}{\longmapsto}y_g^{-1}z_{g-1}z_{g-1}^{-1}y_gx_gy_g^{-1}z_{g-1}^{-1}y_g=x_gy_g^{-1}z_{g-1}^{-1}y_g=y_g^{-1}x_{g-1}y_g;\\
x_g&\stackrel{a_{g-1}^{-1}}{\longmapsto}x_g\stackrel{b_g}{\longmapsto}x_gy_g\stackrel{a_g}{\longmapsto}x_gy_gx_g^{-1}\stackrel{w_{g-1}}{\longmapsto}x_gz_{g-1}^{-1}y_gx_g^{-1}\\
&\stackrel{b_g}{\longmapsto}x_gy_gy_g^{-1}z_{g-1}^{-1}y_gy_g^{-1}x_g^{-1}=x_gz_{g-1}^{-1}x_g^{-1}\\
&\stackrel{a_g}{\longmapsto}x_gz_{g-1}^{-1}x_g^{-1}\stackrel{w_{g-1}}{\longmapsto}x_gz_{g-1}^{-1}x_g^{-1};\\
y_{g-1}&\stackrel{a_{g-1}^{-1}}{\longmapsto}y_{g-1}x_{g-1}\stackrel{b_g}{\longmapsto}y_{g-1}x_{g-1}\stackrel{a_g}{\longmapsto}y_{g-1}x_{g-1}\\
&\stackrel{w_{g-1}}{\longmapsto}y_{g-1}z_{g-1}z_{g-1}^{-1}y_gx_gy_g^{-1}=y_{g-1}y_gx_gy_g^{-1}\\
&\stackrel{b_g}{\longmapsto}y_{g-1}y_gx_gy_gy_g^{-1}=y_{g-1}y_gx_g\stackrel{a_g}{\longmapsto}y_{g-1}y_gx_g^{-1}x_g=y_{g-1}y_g\\
&\stackrel{w_{g-1}}{\longmapsto}y_{g-1}z_{g-1}z_{g-1}^{-1}y_g=y_{g-1}y_g;\\
y_g&\stackrel{a_{g-1}^{-1}}{\longmapsto}y_g\stackrel{b_g}{\longmapsto}y_g\stackrel{a_g}{\longmapsto}y_gx_g^{-1}\stackrel{w_{g-1}}{\longmapsto}z_{g-1}^{-1}y_gx_g^{-1}\\
&\stackrel{b_g}{\longmapsto}y_g^{-1}z_{g-1}^{-1}y_gy_g^{-1}x_g^{-1}=y_g^{-1}z_{g-1}^{-1}x_g^{-1}\\
&\stackrel{a_g}{\longmapsto}x_gy_g^{-1}z_{g-1}^{-1}x_g^{-1}\stackrel{w_{g-1}}{\longmapsto}x_gy_g^{-1}z_{g-1}z_{g-1}^{-1}x_g^{-1}=x_gy_g^{-1}x_g^{-1};\\
z_{g-1}&\stackrel{a_{g-1}^{-1}}{\longmapsto}z_{g-1}\stackrel{b_g}{\longmapsto}z_{g-1}y_g\stackrel{a_g}{\longmapsto}z_{g-1}y_gx_g^{-1}\stackrel{w_{g-1}}{\longmapsto}z_{g-1}z_{g-1}^{-1}y_gx_g^{-1}=y_gx_g^{-1}\\
&\stackrel{b_g}{\longmapsto}y_gy_g^{-1}x_g^{-1}=x_g^{-1}\stackrel{a_g}{\longmapsto}x_g^{-1}\stackrel{w_{g-1}}{\longmapsto}x_g^{-1}.
\end{align*}
\end{proof}

\section{2-categories and their classifying spaces}
We recall definitions and facts of monoidal categories and 2-categories (see \cite{BC03,Hoyo} for more details).

\begin{definition}
A {\it monoidal category} $(\mathcal C, \otimes, I,\alpha, \lambda, \rho)$ is a category $\mathcal C$ equipped with a bifunctor $\otimes:\mathcal C\times\mathcal C\rightarrow \mathcal C$ (called tensor product or monoidal product), an object $I\in\mathcal C$ (called unit object or identity object), and isomorphisms
\begin{align*}
\alpha_{A,B,C}&:(A\otimes B)\otimes C\rightarrow A\otimes(B\otimes C),\\
\lambda_A&:I\otimes A\rightarrow A,\\
\rho_A&:A\otimes I\rightarrow A
\end{align*}
functorial in $A,B,C\in\mathcal C$ (called associator, left and right unitor, respectively) satisfying the following conditions:
\begin{itemize}
\item[(1)] $(A\otimes\alpha_{B,C,D})\alpha_{A,B\otimes C,D}(\alpha_{A,B,C}\otimes D)=\alpha_{A,B,C\otimes D}\alpha_{A\otimes B,C,D}$,
\item[(2)] $(A\otimes\lambda_B)\alpha_{A,I,B}=\rho_A\otimes B$.
\end{itemize}
\end{definition}

A monoidal category is strict if associator, left and right unitor are identities.

\begin{definition}
A {\it $2$-category} consists of:
\begin{itemize}
\item[(1)] A collection of objects ;
\item[(2)] For every pair of objects $A,B\in\mathcal C$, $\mathcal C(A,B)$ form a category: its objects $f,g:A\rightarrow B$ (called {\it $1$-morphisms}) and its morphism $\alpha:f\Rightarrow g$ (called {\it $2$-morphism}), whose composition is denoted by $\circ_v$ and called vertical composition;
\item[(3)] For every object $A\in\mathcal C$, an identity $1$-morphism ${\rm id}_A\in\mathcal C(A,A)$;
\item[(4)] For every triple objects $A,B,C\in\mathcal C$, a composition functor
$$\circ_h:\mathcal C(B,C)\times  \mathcal C(A,B)\rightarrow \mathcal C(A,C),$$
called horizontal composition;
\end{itemize}
These data must satisfy the usual unital and associative axioms.
\end{definition}

Each hom-set in a $2$-category carries a structure of a category. So a $2$-category is a category enriched over $\bf Cat$ (the category of small categories).

For composable $2$-morphisms $\alpha,\beta,\gamma,\delta$, we have an interchange law (See Fig. \ref{fig17}):
$$(\alpha\circ_h\beta)\circ_v(\gamma\circ_h\delta)=(\alpha\circ_v\gamma)\circ_h(\beta\circ_v\delta)$$
\begin{figure}[ht]
  \includegraphics[width=0.13\textwidth,angle=-90]{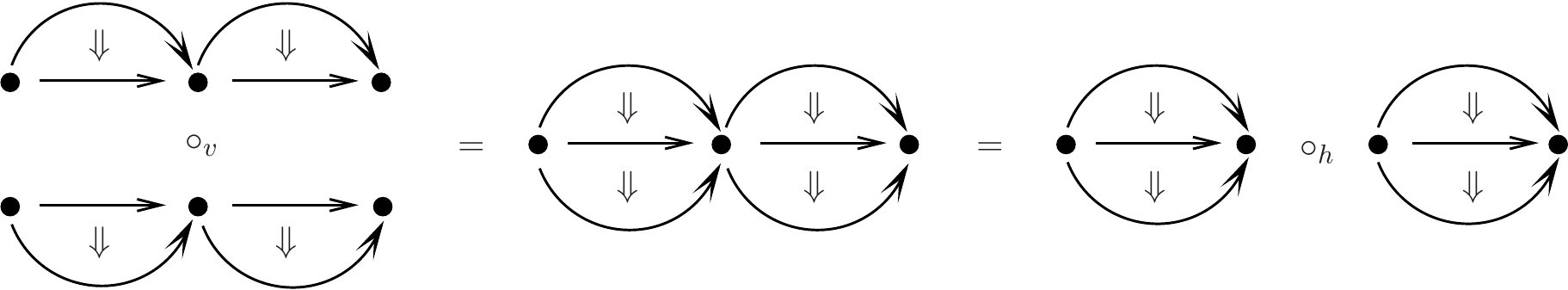} \\
  \caption{}\label{fig17}
\end{figure}

Every (small) category $\mathcal C$ can be regarded as a $2$-category whose only $2$-morphisms are identities. Every (small strict) monoidal category $(\mathcal C,\otimes)$ can be regarded as a $2$-category with a single object, one $1$-morphism for each object of $\mathcal C$, one $2$-morphism for each $1$-morphism of $\mathcal C$, and horizontal composition given by $\otimes$.

A 2-functor $F:\mathcal C\rightarrow\mathcal D$ between 2-categories consists of a map $F:\mathcal C\rightarrow\mathcal D$ together with functors $F:\mathcal C(A,B)\rightarrow\mathcal D(F(A),F(B))$ such that all the structure is preserved.

The {\it nerve} $N\mathcal C$ of a small category $\mathcal C$ is the simplicial set whose $n$-simplices are chains:
$A_0\rightarrow A_1\rightarrow\cdots\rightarrow A_n$
of $n$ composable arrows in $\mathcal C$. Degeneracies in $N\mathcal C$ insert an identity, the first and last faces drop an arrow, and the other faces compose two consecutive ones. The {\it classifying space} $B\mathcal C=|N\mathcal C|$ is the geometric realization of the nerve. It is a CW-complex with one 0-cell for each objects of $\mathcal C$, one 1-cell for each arrow, one 2-cell for each commutative triangle, and so on.
These constructions are functorial so that we have functors
$N:{\bf Cat}\rightarrow{\bf Sim}$, $ B:{\bf Cat}\rightarrow{\bf Top}$,
where {\bf Cat}, {\bf Sim} and {\bf Top} are the categories of small categories, simplicial sets and topological spaces, respectively.

\begin{definition}
Given a 2-category $\mathcal C$, the {\it nerve} $\underline{N}\mathcal C$ is the simplicial category defined by
$$\underline{N}C_n=\coprod_{A_0,\ldots,A_n\in\mathcal C}\mathcal C(A_0,A_1)\times\cdots\times\mathcal C(A_{n-1},A_n).$$
By applying the nerve functor {\bf Cat}$\rightarrow${\bf Sim} in each degree we get a bisimplicial set, which we call the 2-nerve.

The classifying space $B_2\mathcal C$ of $\mathcal C$ is the geometric realization of this bisimplicial set.
\end{definition}

For a 2-category $\mathcal C$, we may apply  nerve construction to the morphism categories. Since the nerve construction commutes with products, we have a category  $\mathcal B \mathcal C$ enriched over $\Delta$-sets, the category of simplicial sets. In other word, $\mathcal B \mathcal C$ is a simplicial category with constant objects. Applying the nerve construction to $\mathcal B \mathcal C$ yields the bisimplicial set which is same as that given in the above definition. The geometric realization of this, denoted by $B_2\mathcal C$. is the classifying space of $\mathcal C$. We have
$$B_2\mathcal C = B \mathcal B \mathcal C.$$

We will simply write $B_2 \mathcal C = B\mathcal C$ if there is no confusion. Note that if a category $\mathcal C$ is regarded as a 2-category, then its classifying space as a 2-category is homeomorphic to the usual $B\mathcal C$. If $\mathcal M$ is a monoidal category, then its classifying space $B\mathcal M$ as a 2-category is the classifying space of the topological monoid $B\mathcal M$. Given a monoidal category $\mathcal M$, we denote by $\overline{\mathcal M}$ its associated 2-category.

\begin{theorem}[\cite{Hoyo}]
Let $(\mathcal M,\otimes)$ be a monoidal category. If $\pi_0(B\mathcal M)$ is a group with product induced by $\otimes$, then
$$B\mathcal M\simeq \Omega B_2\overline{\mathcal M}.$$
\end{theorem}
\vskip 0.2cm
In the next section, for a monoidal 2-category $\mathcal C$, the group completion of its classifying space is denoted by $\Omega B \mathcal C.$

\section{Acyclic embedding and categorical delooping}\label{sec:Braid-Relation}

The map $$\psi:B_g\rightarrow\Gamma_{g,1}$$
which maps the obvious generators $\beta_1, \ldots, \beta_{g-1}$ of $B_g$ to the pillar switchings $\sigma_1, \ldots , \sigma_{g-1}$ is induced by extending the obvious embedding $B_g \hookrightarrow \Gamma_{0,(g+1),1}$ to $B_g \rightarrow \Gamma_{g,2}.$

According to Theorem 2.2, the map $\psi:B_g\rightarrow\Gamma_{g,1}$ may be expressed in terms of Dehn twists as follows:
\begin{align*}
\psi(\beta_i)&=a_{i+2}^{-1} a_{i+1} b_{i+1} w_{i+1} w_{i} a_{i}^{-1}b_{i+1} a_{i+1},\ \  i=1,\ldots,g-2,\\
\psi(\beta_{g-1})&=(w_{g-1}a_gb_g)^2a_{g-1}^{-1}.\\
\end{align*}

We show that $\phi$ embeds $B_g$ into $\Gamma_{g,1}.$

\begin{theorem}
$\psi:B_g\rightarrow\Gamma_{g,1}$ is a monomorphism.
\end{theorem}
\begin{proof}
Artin (\cite{Artin1}) identified $B_n$ as a subgroup of ${\rm Aut}F_n$ as follows.  Artin's map
$$\varphi:B_n\rightarrow {\rm Aut}F_n$$
is defined by mapping $\beta_i$ to the automorphism
$$
\varphi(\beta_i):\left\{
\begin{array}{ll}
\alpha_i\rightarrow \alpha_{i+1} \\
\alpha_{i+1}\rightarrow \alpha_{i+1}^{-1}\alpha_i\alpha_{i+1} .
\end{array}\right.
$$

 Consider the subgroup $F_{\{z_1, \ldots , z_g \}}$ of $\pi_1S_{g,1}=F_{\{y_1, \ldots, z_g, z_1, \ldots , z_g \}}$. According to Corollary 2.1, each $\sigma_i$ acts on $F_{\{z_1, \ldots , z_g \}}$ as

 $$
\sigma_i:\left\{
\begin{array}{ll}
z_i\rightarrow z_{i+1} \\
z_{i+1}\rightarrow z_{i+1}^{-1}z_iz_{i+1} .
\end{array}\right.
$$

Hence the map $$\psi' : B_g \rightarrow {\rm Aut}F_{\{z_1, \ldots , z_g \}}, \quad \beta_i \mapsto \sigma_i\mid_{F_{\{z_1, \ldots , z_g \}}}$$
 is an Artin map, so is injective. If the kernel of the map $$\psi : B_g \rightarrow {\rm Aut}F_{\{y_1, \ldots, y_g, z_1, \ldots , z_g \}}, \quad \beta_i \mapsto \sigma_i$$
 is nontrivial, then so is the kernel of $\psi'$. Therefore, $\psi$ is injective.   \hfill\qed
\end{proof}

We are going to show that  the map $$\psi : B_g\rightarrow\Gamma_{g,1}, \quad \beta_i\mapsto\sigma_i$$ induces a trivial homology homomorphism in the stable range, in the integral coefficient.

\begin{theorem}
The homomorphism $\psi_*  : H_*(B_{\infty};\mathbb Z)\rightarrow H_*(\Gamma_{\infty,1};\mathbb Z)$ induced by $\psi$ is trivial.
\end{theorem}

The remaining part of this paper is the proof of this theorem. As we mentioned in the introduction, we use categorical delooping technique. We should construct a suitable monoidal 2-functor $\Psi:\mathcal T\rightarrow \mathcal S$ which extends $\psi$.

\begin{definition}
Tile category $\mathcal T$  is a strict monoidal 2-category which is a kind of cobordism category between intervals.  Its objects are natural numbers $n\ge 0$ and $n$ represents $n$ ordered copies of unit interval. Its 1-morphisms are generated by three atomic tiles $D:0\rightarrow 1$, $P:2\rightarrow 1$ and $F:1\rightarrow 1$ (See Fig.~\ref{atom-tile}).
\begin{figure}[ht]
  \includegraphics[width=0.45\textwidth]{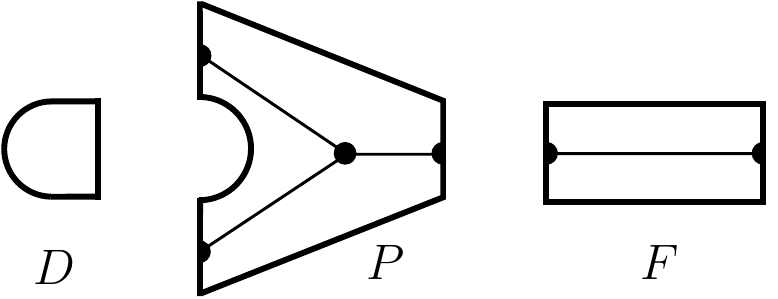}\\
  \caption{Atomic tiles of $\mathcal T$}\label{atom-tile}
\end{figure}
\end{definition}

1-morphisms are generated by gluing along incoming and outgoing intervals, and disjoint union. Note that all incoming and outgoing intervals are ordered. A typical 1-morphism is illustrated in Fig.~\ref{tile-1-morphism}. The composition of 1-morphisms is defined by gluing. Monoid structure is give by disjoint union. Two homeomorphic tiles are not identified in general.

\begin{figure}[ht]
  \includegraphics[width=0.5\textwidth]{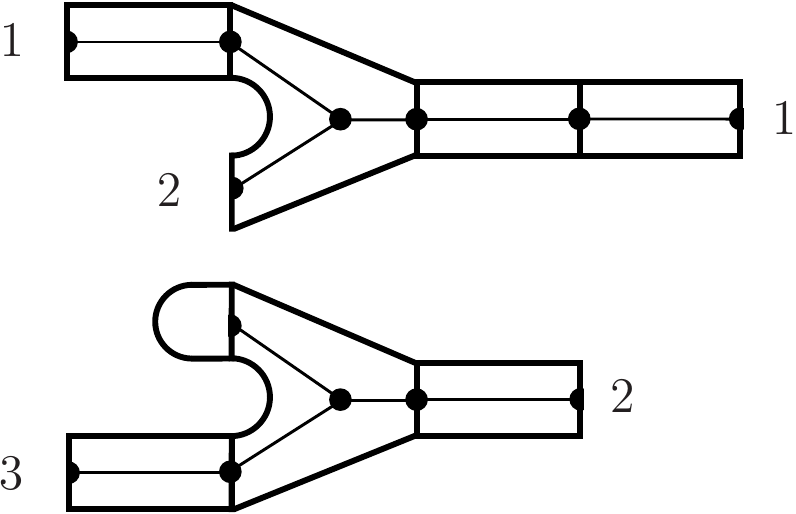}\\
  \caption{A 1-morphism from 3 to 2}\label{tile-1-morphism}
\end{figure}

A 1-morphism constructed by gluing of tiles and disjoint union also forms a disjoint union of trees. A 1-morphism in $\mathcal T(m,n)$ is a disjoint union of $n$ trees. Note that the atomic tiles $P$ and $F$ have half-vertices on the left and right boundary intervals. If a left half-vertex and a right half-vertex meet by gluing, then they form a {\it vertex} of a tree. Ignore half-vertices in a tree unless they are made into complete vertices by meeting their half-pairs.

The morphism categories $\mathcal T(m,n)$ are disjoint union of groups. There are no 2-morphisms between two distinct tiles, that is, $\mathcal T(T_1,T_2)=\varnothing$ if $T_1 \ne T_2$. Given a tile (1-morphism), marked points and edges on it form a graph $\Sigma_T$ as indicated in Fig.~\ref{tile-1-morphism}.

Any half edges in $\Sigma_T$, with half-vertices at the end, that are not completed through gluing are not considered as a part of $\Sigma_T$. For a tile $T$, $\mathcal T(T,T)$ is defined to be the Artin group corresponding to the graph $\Sigma_T$ which is generated by edges of $\Sigma_T$. For two edges $e,f$ in $\Sigma_T$, we have
\begin{align*}
ef&=fe\quad\ \, \text{ if } e\cap f =\varnothing,\\
efe&=fef\quad\text{ if } e\cap f \ne \varnothing.
\end{align*}

For example, let $F^k=F\circ F\circ \cdots \circ F$, then $\mathcal T(F^k,F^k)$ is the braid group $B_{k-2}$. The monoidal structure of $\mathcal T$ is induced by disjoint union.

We define an auxiliary category $\hat{\mathcal T}$ whose objects and 1-morphisms are same as in $\mathcal T$. The morphism categories $\hat{\mathcal T}(m,n)$ are groupoid. For tiles $T_1,T_2$ in $\hat{\mathcal T}(m,n)$, a 2-morphism is an isotopy class of homeomorphisms from $T_1$ to $T_2$ that identify the ordered incoming and outgoing boundary intervals, and map marked points (vertices) bijectively to marked points. Hence $\hat{\mathcal T}(T_1,T_2)=\varnothing$ if either $T_1$ and $T_2$ are not homeomorphic or have the same number of marked points. Note that for a tile $T$ with $k$ marked points is the mapping class group of $T$, that is,  $\hat{\mathcal T}(T,T)=\Gamma(T)=B_k$. There exists an obvious functor $\Theta:\mathcal T\rightarrow\hat{\mathcal T} $.
For a tile $T$, $\Theta$ induces a group homomorphism
$$ \Theta: A(\Sigma_T)\rightarrow \Gamma(T),$$
where $A(\Sigma_T)$ denotes the Artin group induced by graph $\Sigma_T$.

Both $\mathcal T$ and $\hat{\mathcal T}$ contain the subcategory $\mathcal M$ with one object, 1. 1-morphisms in $\mathcal M$ are generated by only one atomic tile $F$.
The classifying space of $\mathcal M$ is $\coprod_{n\ge 0}BB_{n}$. By applying the group completion theorem, we have

\begin{proposition}
$$\Omega B\mathcal M \simeq \mathbb Z\times BB_{\infty}^+.$$
Here $+$ means Quillen's plus construction.
\end{proposition}

By applying the generated group completion theorem via homology fibrations we can determine the homotopy type of classifying space of $\hat{\mathcal T}$.

\begin{proposition}
$\Omega B\hat{\mathcal T}\simeq\mathbb Z\times BB_{\infty}^+$.
\end{proposition}

Surface category $\mathcal S$ (cf. \cite{Tillmann2}, \cite{Tillmann3}), which is slightly modified in this paper,  is a categorical delooping of $B\Gamma_{\infty,1}^+$. Its objects are natural numbers $n\ge 0$, and $n$ represents $n$ ordered disjoint circles. 1-morphisms are generated by three atomic surfaces; a disk $D':0\rightarrow 1$, a surface of genus 2 with two incoming and one outgoing boundary component $P':2\rightarrow 1$ and a torus with one incoming and one outgoing boundary component $F':1\rightarrow 1$.

\begin{figure}[ht]
  \includegraphics[width=0.45\textwidth]{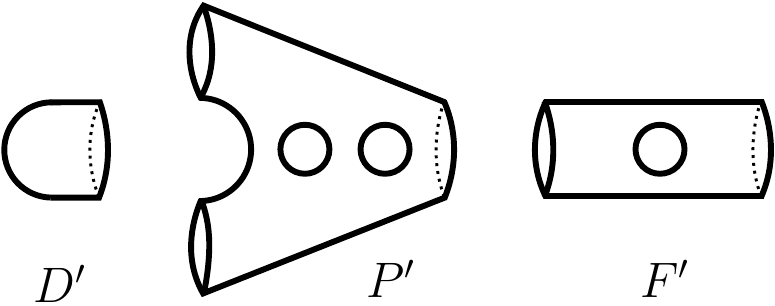}\\
  \caption{Atomic surfaces of $\mathcal S$}\label{atom-tile-surface}
\end{figure}

1-morphisms are generated by gluing incoming boundary circles of one surface to outgoing boundary circles of another surface, and by disjoint union. In addition reordering of the incoming and outgoing circles gives rise to new 1-morphisms. Identity 1-morphisms are adjoint and may be thought as zero length cylinders, A typical 1-morphism  of $\mathcal S$ is illustrated in Fig. \ref{atom-tile-surface-example}.

\begin{figure}[ht]
  \includegraphics[width=0.55\textwidth]{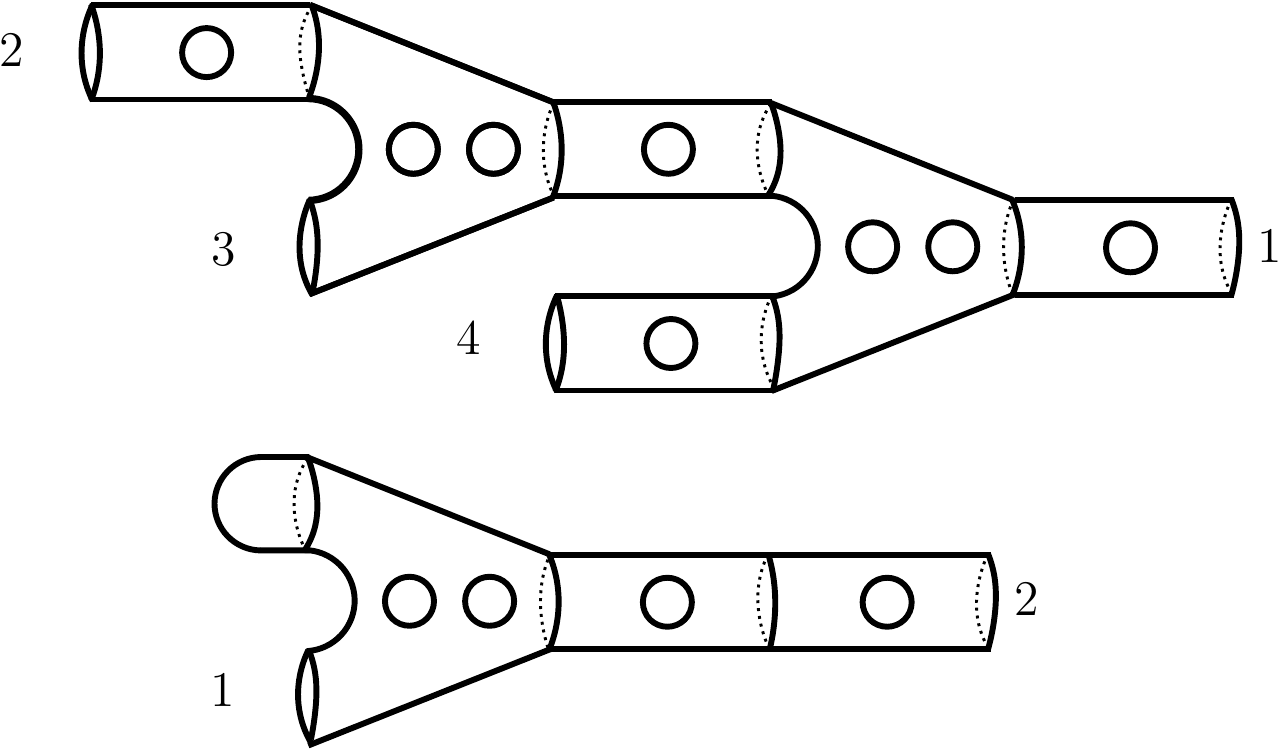}\\
  \caption{A typical 1-morphism of $\mathcal S$}\label{atom-tile-surface-example}
\end{figure}

The set of 2-morphisms in $\mathcal S(m,n)$ between two 1-morphisms $S_1, S_2$ is the set of isotopy classes of homeomorphism from $S_1$ to $S_2$ which identify the $m$ incoming and $n$ outgoing boundary circles. For a 1-morphism $S$, then $\mathcal S(S,S)$ is the mapping class group of $S$.

The disjoint union makes $\mathcal S$ a strict symmetric monoidal category, hence the group completion of its classifying space is an infinite loop space. Let $\mathcal M_{\Gamma}$ be the subcategory of $\mathcal S$ with only one object 1 whose 1-morphisms are generated by the atomic surface $F'$ only. The group completion argument and the homology stability theorem of mapping class groups proves the following.

\begin{proposition}[{\cite[Theorem~3.1]{Tillmann2}}]
$\Omega B\mathcal M_{\Gamma}\simeq \Omega B\mathcal S \simeq \mathbb Z\times B\Gamma_{\infty,1}^+$.
\end{proposition}

We now define the monoidal 2-functor $\Psi:\mathcal T\rightarrow \mathcal S$. $\Psi$ maps $n$ intervals to $n$ circles. For atomic tiles, $\Psi$ is defined by
$$\Psi(D)=D',~\Psi(P)=P',~\Psi(F)=F'.$$

\begin{figure}[ht]
  \includegraphics[width=0.4\textwidth]{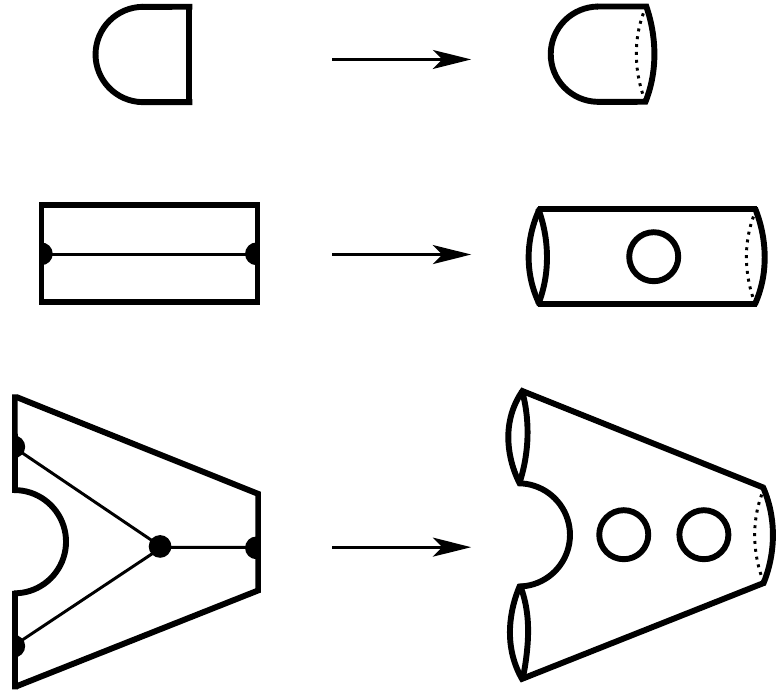}\\
  \caption{The functor $\Psi$ on atomic 1-morphisms }\label{Tile-functor-map}
\end{figure}

$\Psi$ extends to a map of all 1-morphism as 1-morphisms are build in $\mathcal T$ and $\mathcal S$ from atomic 1-morphisms in the same fashion. The assignment of $\Psi$ on 1-morphisms is totally geometric and obviously functorial.

\begin{figure}[ht]
  \includegraphics[width=0.5\textwidth]{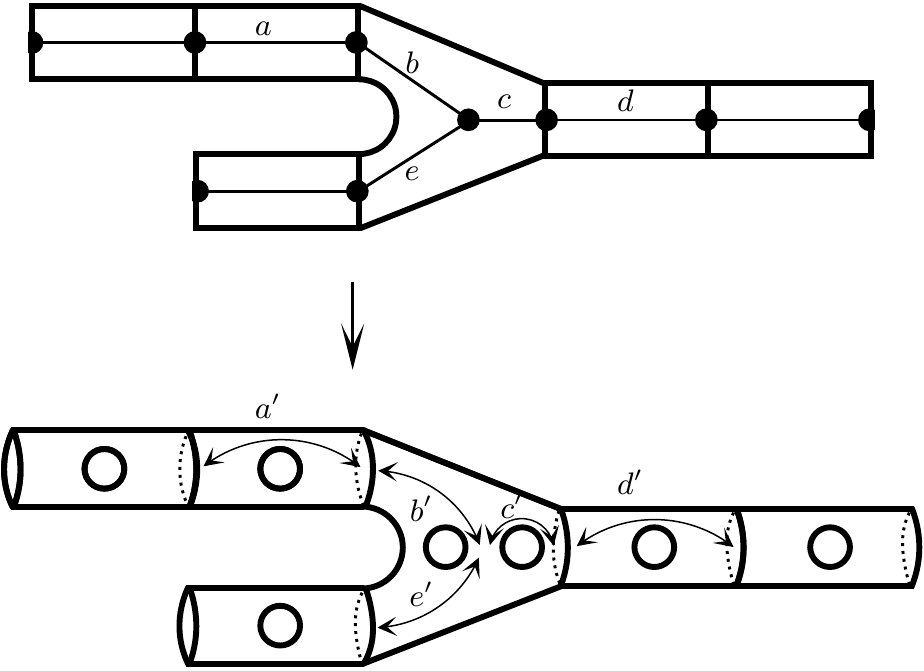}\\
  \caption{The functor $\Psi$ on a 2-morphism }\label{tile-surface-example}
\end{figure}

\begin{figure}[ht]
  \includegraphics[width=0.5\textwidth]{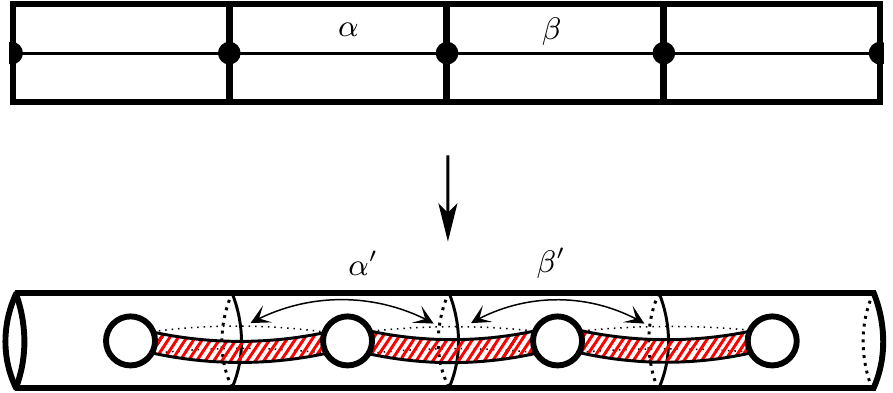}\\
  \caption{The functor $\Psi$ on a 2-morphism }\label{braid-pillar-map}
\end{figure}

To define the functor $\Psi$ on 2-morphisms, consider a graph $\Sigma_T$ corresponding to a tile $T$ and the Artin group $A(\Sigma_T)$ generated by edges of $\Sigma_T$. To each full edge (ignoring half edges) of $\Sigma_T$, assign a pillar switching as in Fig.~\ref{tile-surface-example} and \ref{braid-pillar-map}. We may think that to each vertex of $\Sigma_T$, $\Psi$ assigns a pillar of surface in $\mathcal S$, and ignore half-vertices. Then we may think that to each edge (generator of the group $A(\Sigma_T)$) $\Psi$ assigns a pillar switching.  The map of generators $e \mapsto \sigma$ extends to a well-defined group homomorphism
$$\Psi : \mathcal T (T, T) =A(\Sigma_T ) \rightarrow \mathcal S (\Psi(T), \Psi(T)) = \Gamma(\Psi(T)).$$

Then this assignment is functorial under gluing and $\Psi$ commutes with disjoint union, hence $\Psi$ is a strict monoidal 2-functor. Then we have
$$\Omega B\Psi:\Omega B\mathcal T\rightarrow\Omega B\mathcal S.$$
is a map of double loop spaces.

\begin{proposition}
The map $\psi^+:BB_{\infty}^+\rightarrow B\Gamma_{\infty,1}^+$ is a map of double loop spaces.
\end{proposition}

\begin{proof}
By Theorem 4.1 of \cite{Song-Tillmann} the inclusion functor $J: \mathcal M \hookrightarrow \mathcal T$ induces a double loop space map $\Omega BJ : \Omega B\mathcal M \rightarrow \Omega B \mathcal T$. Let $\Omega_0 B \mathcal M$ and $\Omega_0 B \mathcal T$ denote base point components of $\Omega B \mathcal M$ and $\Omega B \mathcal T$, respectively. The map $\psi^+:BB_{\infty}^+\rightarrow B\Gamma_{\infty,1}^+$ is regarded as the composite

$$BB_{\infty}^+ = \Omega_0 B \mathcal M \stackrel{\Omega BJ}{\longrightarrow} \Omega_0 B \mathcal T \stackrel{\Omega B\Psi}{\longrightarrow}\Omega_0 B\mathcal S \simeq \Omega_0 B \mathcal M_{\Gamma} = B\Gamma_{\infty}^+.$$
Hence it is a map  of double loop spaces.
\end{proof}

Since every double loop space map $BB_{\infty}^+\rightarrow B\Gamma_{\infty,1}^+$  is homotopically trivial (\cite{Song-Tillmann}, Lemma 5.3) and the plus construction does not change the homology, Theorem 4.2 has been proved.

\begin{remark}
By the homology stability theorem we have the following unstable version:  $$\psi_*:H_i(B_g;\mathbb Z) \rightarrow H_i(\Gamma_{g,1};\mathbb Z)$$ is zero for $0<i< 2g/3$.
\end{remark}

\begin{acknowledgements}
This research was supported by Basic Science Research Program through the National Research Foundation of Korea (NRF) funded by the Ministry of Education, Science and Technology (2011-0004509).
\end{acknowledgements}

\end{document}